\documentclass[reqno,12pt,a4letter]{amsart}
\usepackage{amsmath, amsxtra, amssymb, latexsym, amscd, amsthm,enumerate}
\usepackage{graphicx, color}
\usepackage[active]{srcltx}
\usepackage[utf8]{inputenc}
\usepackage[mathscr]{euscript}
\usepackage{mathrsfs,cite}
\usepackage[english]{babel}
\usepackage{color}
\usepackage[active]{srcltx}
\def\Limsup{\mathop{{\rm Lim}\,{\rm sup}}}

\def\dom{\mbox{\rm dom}\,}

\def\R{\mathbb{R}}
\def\N{\mathbb{N}}

\newtheorem {theorem}{Theorem}[section]
\newtheorem {corollary}{Corollary}[section]
\newtheorem {proposition}{Proposition}[section]
\newtheorem {lemma}{Lemma}[section]
\newtheorem {example}{Example}[section]
\newtheorem {definition}{Definition}[section]
\newtheorem {remark}{Remark}[section]

\title{On error bounds and optimality conditions at infinity}

\author{NGUYEN VAN TUYEN$^{1}$}
\address{$^1$Department of Mathematics, Hanoi Pedagogical University 2, Xuan Hoa, Phuc Yen, Vinh Phuc, Vietnam}
\email{nguyenvantuyen83@hpu2.edu.vn; tuyensp2@yahoo.com}

\date{\today}



\begin{document}
	
	\maketitle

{\centerline{\em Dedicated to Professor Do Sang Kim on the occasion of his 70th birthday}}
	\begin{abstract}
In this paper, we establish sufficient  conditions for the existence of error bounds  at infinity for lower semicontinuous  inequality systems. We also show that the existence of an error bound at infinity of  constraint systems plays an important role in deriving necessary optimality conditions at infinity for constrained optimization problems.     
	\end{abstract}
	
\section{Introduction}\label{Introduction}
Consider the inequality system
\begin{equation}\label{gerenal-constraint-system}
	S:=\{x\in \Omega\subset X\;:\; g(x)\leq 0\}
\end{equation}
where $X$ is a Banach space, $\Omega$ is a nonempty and closed subset in $X$, and $g\colon X\to\overline{\mathbb{R}}$ is a lower semicontinuous function.

We say that $S$ has a local error bound at a point $\bar x\in S$ if there $\tau>0$ and a neighborhood $V$ of $\bar x$ such that
\begin{equation}\label{local-error-bound}
	d(x; S)\leq \tau [g(x)]_+\  \ \forall x\in \Omega\cap V,
\end{equation}  
where $d(x; S)$ denotes the distance from $x$ to $S$ and $[g(x)]_+:=\max\{g(x), 0\}$. If we can choose $V=X$ in \eqref{local-error-bound}, then we say that $S$ has a global/Hoffman error bound.

The concept of error bound plays a center role in  many fields of mathematical programming theory including, for example, in analyzing the convergence of algorithms, in studying optimality conditions,  stability, sensitivity,  subdifferential calculus, and so on. Today, the literature on error bounds is rich and there are many conditions that ensure the existence of error bounds, see, for example \cite{Auslender,Dinh,Ha,Hang,Hoffman_52,Ioffe-79,Jourani-2000,Gli-15,Gli-18,Luo,Ngai-2005,Ngai-2009,Pang-1997,Robinson-1975,Robinson-1979,Wei,Wei-2020}. 

To the best of our knowledge, the first result on global error bounds was established by Hoffman \cite{Hoffman_52}. The author proved that if $g$ is  a maximum of a finite number of affine functions in $\R^n$, then $S$ has  a global error bound. In \cite{Robinson-1979}, Robinson extended the Hoffman's result to bounded convex differentiable inequality systems in infinite dimensional spaces. The existence of global error bound for unbounded convex inequality systems in reflexive Banach spaces was studied by Deng \cite{Deng}. In \cite{Ioffe-79}, Ioffe presented a sufficient condition 
for the existence of a local error bound for nonconvex and nondifferentiable Lipschitz continuous equality system in infinite dimension.  The relationship between the existence of local error bounds and Abadie constraint qualification conditions for convex inequality systems  can be found in \cite{Jourani-2000,Ngai-2005,Ngai-2009,Wei,Wei-2020}. Regarding to the existence of H\"older error bounds, we refer the readers to references \cite{Dinh,Ha,Gli-15,Luo,Pang-1997,Wang}.  

In mathematical programming, there are numerous problems where the objective function is bounded from below but nevertheless fails to attain its minimum at any finite feasible solution; see, for example, \cite{Dudik}.  When this happens, the Fermat's rule cannot be applied. However, there exists a sequence of feasible solutions tending to infinity such that  the corresponding sequence of objective values tends to its infimum. Therefore, optimality conditions in this case must somehow be taken  {\em``at infinity''}. To do this, very recently, Nguyen and Pham \cite{Tung-Son-22} introduced the notions  of Clarke's tangent,  normal cones, subgradients, Lipschitz continuity, ..., at infinity and established necessary optimality conditions at infinity. Thereafter, Kim, Nguyen and Pham \cite{Kim-Tung-Son-23} proposed the concepts of limiting normal cones at infinity to unbounded sets as well as limiting and singular subdifferentials at infinity and then they derived necessary optimality conditions as well as sufficient conditions for the weak sharp minima property at infinity. To obtain necessary optimality conditions at infinity, the authors used a constraint qualification of Mangasarian--Fromovitz type at infinity for Lipschitz at infinity functions. We note here that both assumptions on the Lipschitzness and the Mangasarian--Fromovitz constraint qualification  at infinity are rather strict.    

To the best of our knowledge, so far there have been no papers studying the concept of error bound at infinity. In this paper, we introduce the concept of error bound at infinity for general constraint systems and derive sufficient conditions for such ones have an error bound at infinity. Then we show that the existence of an error bound at infinity is essential to derive upper estimate for the normal cone at infinity of constraint systems. As a byproduct, we obtain new necessary optimality conditions at infinity  for  problems which do not attain its minimum.

The rest of the paper is organized as follows. Section \ref{Preliminaries} contains some  definitions and preliminary results from variational analysis and generalized differentiation. In Section \ref{Section 3}, we derive sufficient conditions for error bounds at infinity.  Section \ref{Section 4} is devoted to upper  estimates of the normal cone of constraint systems and optimality conditions at infinity. 

\section{Preliminaries}\label{Preliminaries}
In this section, we recall some notions related to generalized differentiation from \cite{Kim-Tung-Son-23,Mordukhovich2006,Mordukhovich2018,Penot-2013,Rockafellar1998}. The  space $\R^n$ is equipped with the usual scalar product $\langle \cdot, \cdot\rangle$ and the corresponding  Euclidean norm $\|\cdot\|$. The closed unit ball   and the nonnegative orthant  in $\R^n$ are denoted, respectively, by $\mathbb{B}$ and $\R^n_+$.  The   convex hull of a set $D\subset \R^n$ and the  Euclidean projector of a point $x\in\R^n$ to $D$  are  denoted, respectively, by  $\mathrm{co}\, D$ and $\Pi_D(x)$. As usual, we denote $[\alpha]_+:=\max\,\{\alpha, 0\}$ for any $\alpha\in\R$,  $\overline{\mathbb{R}}:=\R\cup\{\infty\}$, and  $d(x; D)$ is the distance of a point $x\in\R^n$ to $D$, that is, 
$$d(x; D):=\inf\{\|y-x\|\;:\; y\in D\}.$$
Let  $F : \R^n \rightrightarrows \R^m$ be a set-valued mapping. The \textit{domain} and the \textit{graph}  of $F$ are given, respectively, by
$${\rm dom}\,F=\{x\in \R^n \mid F(x)\not= \emptyset\} $$
and
$$ {\rm gph}\,F=\{(x,y)\in \R^n \times \R^m \mid y \in F(x)\}.$$
The set-valued mapping is called  $F$ \textit{proper} if $\dom F \not= \emptyset.$ The \textit{Painlev\'e-Kuratowski outer/upper limit} of $F$ at $\bar x$ is defined by
\begin{align*} 
	\Limsup\limits_{x\rightarrow \bar x} F(x):=\bigg\{ y\in \mathbb{R}^m \mid \exists x_k \rightarrow \bar x, y_k \rightarrow y \ \mbox{with}\ y_k\in F(x_k), \forall k=1,2,....\bigg\}.
\end{align*}
\subsection{Normal cones and Subdifferentials}
\begin{definition}[{see \cite{Mordukhovich2006,Mordukhovich2018}}]{\rm 
		Let $\Omega$ be a nonempty subset of $\mathbb{R}^n$ and $\bar x \in \Omega$. 
		\begin{enumerate}[(i)]
			\item The \textit{regular/Fr\'echet normal cone} to $\Omega$ at $\bar x$ is defined by
			\begin{align*}
				\widehat N(\bar x; \Omega)=\left\{ v\in \mathbb{R}^n\mid \limsup\limits_{x \xrightarrow{\Omega}\bar x} \dfrac{\langle v, x-\bar x \rangle}{\|x-\bar x\|} \leq 0 \right\},
			\end{align*}
			where $x \xrightarrow{\Omega} \bar x$ means that $x \rightarrow \bar x$ and $ x\in \Omega$.
			\item The \textit{limiting/Mordukhovich normal cone} to $\Omega$ at $\bar x$ is given by
			\begin{align*}
				N(\bar x; \Omega)=\Limsup\limits_{ x \xrightarrow{\Omega} \bar x} \widehat{N}(x; \Omega).
			\end{align*}
			When $\bar x \not\in \Omega$, we put $\widehat N(\bar x;\Omega)=N(\bar x;\Omega) =\emptyset$.
	\end{enumerate} }
\end{definition}	
By definition, it is clear that
\begin{align*} 
	\widehat N(x;\Omega) \subset N(x;\Omega),\ \ \   \forall  x \in \Omega.
\end{align*}
Let $f: \mathbb{R}^n\rightarrow \overline{\mathbb{R}}$ be an extend real-valued function. The \textit{effective domain} and the \textit{epigraph} of $f$ are denoted, respectively, by
$$\mbox{dom}\, f:=\{x \in \mathbb{R}^n \mid f(x) < +\infty\}$$ 
and  
$$ {\rm{epi}}\, f:=\{ (x, \alpha) \in \mathbb{R}^n \times \mathbb{R} \mid \alpha \ge f (x)\}.$$  
We say that $f$ is {\em proper}  if its $\mathrm{dom} f$ is  nonempty.

\begin{definition}[{see \cite{Mordukhovich2006,Mordukhovich2018}}]\label{def21} {\rm Consider a function  $f \colon \mathbb{R}^n \to \overline{\mathbb{R}}$ and a point $\bar{x} \in {\rm dom}f$.  
		\begin{enumerate}[{\rm (i)}]
			\item The {\em regular/Fr\'echet subdifferential} of $f$ at $\bar{x}$ is defined by 
			$$
			\widehat{\partial}f(\bar{x}):=\{ v \in \mathbb{R}^n \mid (v,-1)\in \widehat{N}((\bar{x},f(\bar{x}));\mathrm{epi} f)  \}.  
			$$
			\item The {\em limiting/Mordukovich subdifferential} and the {\em limiting/Mordukovich singular subdifferential}  of $f$ at $\bar{x}$ are defined, respectively, by  
			$$
			\partial f(\bar{x}):=\{ v \in \mathbb{R}^n \mid (v,-1)\in {N}((\bar{x},f(\bar{x}));\mathrm{epi} f)  \}, 
			$$
			and   
			$$\partial^{\infty} f(\bar{x}):=\{ v \in \mathbb{R}^n \mid \exists v_k \in 	\widehat{\partial} f(\bar{x}), \lambda_k \downarrow 0, \lambda_k v_k \to v \}.	$$
	\end{enumerate}}
\end{definition}
It is well-known that
$$
\partial f(\bar{x})=\Limsup_{x \xrightarrow{f} \bar{x}}\widehat{\partial}f(x) \supseteq \widehat{\partial}f(x),
$$
where $x \xrightarrow{f} \bar{x}$ means that $x \to \bar{x}$ and $f(x) \to f(\bar{x})$. When the function  $f$ is  convex, then the subdifferentials $\widehat{\partial} f(\bar{x})$ and $\partial f(\bar{x})$ coincide with the subdifferential in the sense of convex analysis. 

For the singular subdifferential, we have
\begin{align*}
	\partial^{\infty} f(\bar{x})\subseteq \{ v \in \mathbb{R}^n \mid (v,0)\in N((\bar{x},f(\bar{x}));\mathrm{epi} f)  \},
\end{align*}
and the inclusion holds with equality whenever $f$ is locally lower semicontinuous (l.s.c.) at $\bar{x}$; see \cite[Theorem 8.9]{Rockafellar1998}. 

Let $\Omega \subset \mathbb{R}^n$. The {\em indicator function} $\delta_{\Omega} \colon \mathbb{R}^n \to \overline{\mathbb{R}}$ of $\Omega$ is defined by
$$\delta_{\Omega}(x) :=
\begin{cases}
	0 & \textrm{ if } x \in \Omega, \\
	+\infty & \textrm{ otherwise.}
\end{cases}$$
It holds that $\partial \delta_{\Omega} (x) =\partial^{\infty} \delta_{\Omega} (x)= N(x; \Omega)$ for any $x \in \Omega$; see \cite{Mordukhovich2006,Mordukhovich2018,Rockafellar1998}.

The following result is the nonsmooth versions of Fermat's rule.
\begin{lemma}[{see \cite[Proposition 1.114]{Mordukhovich2006}}] \label{lema22}
	If a proper function $f \colon \mathbb{R}^n \to \overline{\mathbb{R}}$ has a local minimum at $\bar{x},$ then $0 \in \widehat{\partial}f(\bar{x})\subset \partial f(\bar{x}).$	
\end{lemma}
The next result gives sum rules for limiting and singular subgradients of extended-real-valued functions.
\begin{lemma}[{see \cite[Theorem 3.36]{Mordukhovich2006}}]\label{sum-rule} Let $f_i\colon\mathbb{R}^n\to\overline{\mathbb{R}}$, $i=1, \ldots, m$, $m\geqq 2$, be l.s.c. around $\bar x$ and let all but one of these functions be locally Lipschitz around $\bar x$. Then we have the following inclusions
	\begin{align*}
		&\partial (f_1+\ldots+f_m) (\bar x)\subset \partial  f_1 (\bar x) +\ldots+\partial f_m (\bar x),
		\\
		&\partial^\infty (f_1+\ldots+f_m) (\bar x)\subset \partial^\infty  f_1 (\bar x) +\ldots+\partial^\infty f_m (\bar x).		
	\end{align*}
\end{lemma}

We now recall the Ekeland variational principle (see \cite{Ekeland-74}).
\begin{lemma}[Ekeland variational principle]\label{lema26} 
	Let $f \colon \mathbb{R}^n \to \overline{\mathbb{R}}$ be a proper l.s.c. function and bounded from below. Let $\epsilon >0$ and $x_0 \in \mathbb{R}^n$ be satisfied
	\begin{eqnarray*}
		f (x_0) &\le& \inf_{x \in \mathbb{R}^n}f (x) + \epsilon. 
	\end{eqnarray*}
	Then, for any $\lambda >0$ there exists  $x_1 \in \mathbb{R}^n$ such that
	\begin{enumerate}[{\rm (i)}]
		\item $f (x_1) \le f(x_0),$
		\item $\|x_1-x_0\| \le \lambda,$ and
		\item $f (x_1) \le  f(x)+\dfrac{\epsilon}{\lambda}\|x-x_1\|$ for all $x \in \mathbb{R}^n.$
	\end{enumerate}
\end{lemma}

\subsection{Normal cones and subdifferentials at infinity}
Let $\Omega$ be a {\em locally closed} subset of $\mathbb{R}^n$, i.e., for any $x\in \Omega$ there is a neighborhood $U$ of $x$ such that $\Omega\cap U$ is closed.  Assume that  $\Omega$ is unbounded. 

\begin{definition}[{see \cite{Kim-Tung-Son-23}}]\label{def31} {\rm 
		The {\em norm cone to the set $\Omega$ at infinity} is defined by
		\begin{eqnarray*}
			N(\infty; \Omega) &:=& \Limsup_{ x \xrightarrow{\Omega} \infty} \widehat{N}(x; \Omega).
		\end{eqnarray*} 
}\end{definition}
The following result is  the intersection rule for
normals at infinity.
\begin{proposition}[{see \cite[Proposition 3.7]{Kim-Tung-Son-23}}]\label{intersection-normal-cones}
	Let $\Omega_1, \Omega_2$ be locally closed subsets of $\mathbb{R}^n$ satisfying the normal qualification condition at infinity
	\begin{eqnarray*}
		N(\infty; \Omega_1) \cap \big (-N(\infty; \Omega_2) \big) &=& \{0\}.
	\end{eqnarray*}
	Then we have the inclusion
	\begin{eqnarray*}
		N(\infty; \Omega_1 \cap \Omega_2)  &\subset& N(\infty; \Omega_1)+ N(\infty; \Omega_2).
	\end{eqnarray*}
\end{proposition}

Now let $f \colon \mathbb{R}^n \to \overline{\mathbb{R}}$ be a l.s.c. function and assume that $f$ is {\em proper at infinity} in the sense that the set $\mathrm{dom} f$ is unbounded.

\begin{definition}[{see \cite{Kim-Tung-Son-23}}]\label{def41} {\rm 
		The {\em limiting/Mordukhovich and the singular subdifferentials} of $f$ at infinity are defined, respectively, by 
		\begin{eqnarray*}
			\partial f(\infty) &:=& \{u \in \mathbb{R}^n \ | \ (u, -1) \in \mathcal{N} \},\\
			\partial^{\infty} f(\infty) &:=& \{u \in \mathbb{R}^n \ | \ (u, 0) \in \mathcal{N}\},
		\end{eqnarray*}
		where $\mathcal{N} := \displaystyle \Limsup_{x \to \infty} {N}((x,f(x)); \textrm{epi} f).$
}\end{definition}
The following result gives limiting representations of limiting and singular subgradients at infinity.
\begin{proposition}[{see \cite[Proposition 4.4]{Kim-Tung-Son-23}}]\label{pro42} 
	The following relationships hold
	\begin{align*}
		\partial f(\infty) &= \Limsup_{x \to \infty} \partial f(x)= \Limsup_{x \to \infty} \widehat{\partial} f(x), 
		\\
		\partial^{\infty} f(\infty) &= \Limsup_{x \to \infty, r \downarrow 0} r \partial f(x)\supseteq \Limsup_{x \to \infty} \partial^{\infty} f(x).  
	\end{align*}
\end{proposition}
\begin{remark}
	{\rm Let $\Omega$ be an unbounded and closed  subset of $\R^n$. Then it follows from Proposition \ref{pro42} and \cite[Proposition 1.19]{Mordukhovich2018} that
		\begin{equation*} 
			\partial\delta_\Omega(\infty)=\partial^\infty\delta_\Omega(\infty)=N(\infty; \Omega).
		\end{equation*}  
		
	}
\end{remark}

We now recall the notion of the Lipschitz property at infinity for l.s.c. functions (see \cite{Kim-Tung-Son-23,Tung-Son-22}). 
\begin{definition}\label{def51}{\rm
		Let  $f \colon \mathbb{R}^n \to \mathbb{R}$ be a l.s.c function.  We say that $f$ is {\em Lipschitz at infinity} if there exist constants $L > 0$ and $R > 0$ such that
		\begin{eqnarray*}
			|f(x) - f(x')| &\le& L \|x - x'\| \ \  \textrm{ for all } \ \ x, x' \in \mathbb{R}^n \setminus \mathbb{B}_R.
		\end{eqnarray*}
}\end{definition}
The following result gives a  necessary and sufficient condition for the  Lipschitz property at infinity of l.s.c. functions. 
\begin{proposition}[{see \cite[Proposition 5.2]{Kim-Tung-Son-23}}]\label{pro52}   
	Let $f \colon \mathbb{R}^n \to \mathbb{R}$ be a l.s.c. function. Then $f$ is Lipschitz at infinity if and only if $\partial^{\infty}f(\infty)=\{0\}.$ In this case, $\partial f(\infty)$ is nonempty compact. 
\end{proposition}
The next results present  calculus rules for both basic and singular  subdifferentials at infinity. These results obtained from \cite{Kim-Tung-Son-23} by induction. 
\begin{proposition}[{cf. \cite[Proposition 4.9]{Kim-Tung-Son-23}}] \label{pro54}
	Let $f_1, \ldots, f_m \colon \mathbb{R}^n\to\overline{\mathbb{R}}$ be l.s.c. functions such that the following   qualification condition holds
	\begin{equation}\label{qualification-condition}
		\left[u_1+\ldots+u_m=0, u_i\in\partial^{\infty} f_i(\infty)\right] \Rightarrow u_i=0, i=1, \ldots, m,
	\end{equation}
	Then we have 
	\begin{eqnarray*}
		\partial (f_1 +\ldots+ f_m)(\infty) &\subset& \partial f_1(\infty) +\ldots+ \partial  f_m(\infty), \\
		\partial^{\infty} (f_1 +\ldots+ f_m)(\infty) &\subset& \partial^{\infty} f_1(\infty) + \ldots+ \partial^{\infty} f_m(\infty).   
	\end{eqnarray*}
\end{proposition}
\begin{proposition}[{cf. \cite[Proposition 4.11]{Kim-Tung-Son-23}}] \label{pro-2.15}
	Let $f_1, \ldots, f_m \colon \mathbb{R}^n\to\overline{\mathbb{R}}$ be l.s.c. functions such that the   qualification condition  \eqref{qualification-condition} is satisfied. Then one has the  inclusions
	\begin{eqnarray*}
		\partial (\max\{f_1, \ldots, f_m\})(\infty) &\subset& \bigcup\left\{\sum_{i=1}^m\lambda_i\circ\partial f_i(\infty)\;:\; \lambda\in \Delta_m\right\}, \\
		\partial^{\infty} (\max\{f_1, \ldots, f_m\})(\infty) &\subset& \sum_{i=1}^m \partial^{\infty} f_i(\infty),   
	\end{eqnarray*}
	where $\Delta_m:=\{\lambda\in\R^m_+\;:\; \sum_{i=1}^m\lambda_i=1\}$ and $\lambda_i\circ \partial f_i(\infty)$, $i=1, \ldots, m$, are defined as follow
	\begin{equation*}
		\lambda_i\circ  \partial f_i(\infty)=
		\begin{cases}
			\lambda_i \partial f_i(\infty) \ \ &\text{if}\ \ \lambda_i>0,
			\\
			\partial^{\infty} f_i(\infty) \ \ &\text{if} \ \ \lambda_i=0.
		\end{cases}
	\end{equation*}
\end{proposition}

\begin{remark}
	{\rm By Proposition \ref{pro52}, the   qualification condition \eqref{qualification-condition} holds if  all but one of functions $f_1, \ldots, f_m$ are Lipschitz at infinity. 
		
	}
\end{remark}
\begin{proposition}[{cf. \cite[Proposition 4.12]{Kim-Tung-Son-23}}] \label{pro-2.16} Let $f_1, \ldots, f_m \colon \mathbb{R}^n\to\overline{\mathbb{R}}$ be l.s.c. functions. Then we have
	\begin{equation*}
		\partial (\min\{f_1, \ldots, f_m\})(\infty)\subset \bigcup \{\partial f_i(\infty)\;:\; i= 1, \ldots, m\}. 
	\end{equation*}
	
\end{proposition}

\section{Sufficient conditions for error bounds at infinity} \label{Section 3}
Let $\Omega$ be a nonempty and closed subset in $\mathbb{R}^n$ and $g\colon\mathbb{R}^n\to\overline{\mathbb{R}}$ be a l.s.c. function. Consider the constraint set
\begin{equation}\label{constraint-set}
	S:=\{x\in \Omega\;:\; g(x)\leq 0\}.	
\end{equation}
Assume that $\Omega\cap\dom g$ is nonempty and unbounded. 

Our main purpose in this section is to derive sufficient conditions for the existence of error bounds at infinity for the constraint set $S$ in \eqref{constraint-set}.  
\begin{definition}\rm 
	We say that the constraint set $S$ has an error bound at infinity if there exist $\alpha>0$ and $R>0$ such that
	\begin{equation*}
		d(x; S)\leq \alpha [g(x)]_+
	\end{equation*}
	for all $x\in \Omega$ with $\|x\|>R$.
\end{definition}
\begin{theorem}\label{theorem-3.2} Assume that 
	\begin{equation}\label{CQ-infinity}
		\partial^\infty g(\infty)\cap(-N(\infty; \Omega))=\{0\}
	\end{equation}
	and 
	\begin{equation}\label{regular-infinity}
		0\notin \partial g(\infty)+N(\infty; \Omega).
	\end{equation}
	Then the constraint set $S$  in \eqref{constraint-set} has an error bound at infinity. 
\end{theorem}
\begin{proof} Suppose on the contrary that $S$ has no error bound at infinity. Then, for each $k\in\mathbb{N}$, there exists $x_k\in \Omega$ such that $\|x_k\|\to \infty$ as $k\to\infty$ and
	\begin{equation}\label{equa-5}
		d(x_k; S)>k^2[g(x_k)]_+.
	\end{equation}  
	This implies that $x_k\notin S$ for all $k\in\N$. Hence, $\epsilon_k:=[g(x_k)]_+=g(x_k)>0$  and 
	\begin{equation*}
		[g(x_k)]_+\leq \inf_{y\in\Omega}[g(y)]_++\epsilon_k
	\end{equation*}
	for all $k\in\mathbb{N}$. By the closedness of $\Omega$ and the Ekeland variational principle (see Lemma \ref{lema26}), for each $k\in\mathbb{N}$ and $\lambda_k=k \epsilon_k>0$, there exists $y_k\in\Omega$ such that
	\begin{align}
		&\|y_k-x_k\|\leq \lambda_k,\label{equ_3.6} 
		\\
		& [g(y_k)]_+\leq [g(y)]_++\frac{1}{k}\|y-y_k\| \ \ \text{for all}\ \ y\in\Omega.\label{equ_3.7}
	\end{align}  
	It follows from \eqref{equa-5} that
	\begin{equation}\label{equa-6}
		\lambda_k<\frac{1}{k} d(x_k; S)\leq d(x_k; S).
	\end{equation}
	This implies that $y_k\notin S$ and so $g(y_k)>0$. Indeed, if otherwise, then 
	\begin{equation*}
		\|y_k-x_k\|\geq d(x_k; S)>\lambda_k,
	\end{equation*}
	a contradiction. Now fixed $\bar x\in S$, then by \eqref{equ_3.6} and \eqref{equa-6}, one has 
	\begin{align*}
		\|y_k\|\geq \|x_k\|-\|y_k-x_k\|\geq \|x_k\|-\lambda_k&>\|x_k\|-\frac{1}{k} d(x_k; S)
		\\
		&\geq \|x_k\|-\frac{1}{k}\|x_k-\bar x\|
		\\
		&=\|x_k\|\left(1-\frac{1}{k}\left\|\frac{x_k}{\|x_k\|}-\frac{\bar x}{\|x_k\|}\right\|\right)
	\end{align*}  
	and so $\|y_k\|\to \infty$ as $k\to\infty$. On the other hand, by \eqref{equ_3.7} $y_k$ is a global minimizer of the l.s.c. function $[g(\cdot)]_++\delta_{\Omega}(\cdot)+\frac{1}{k}\|\cdot-y_k\|$  on $\R^n$. By the l.s.c. property of $g$ and the fact that $g(y_k)>0$ there exists a neighborhood of $y_k$ on which we have
	$$[g(\cdot)]_++\delta_{\Omega}(\cdot)+\frac{1}{k}\|\cdot-y_k\|=g(\cdot)+\delta_{\Omega}(\cdot)+\frac{\epsilon_k}{\lambda_k}\|\cdot-y_k\|.$$
	By Fermat rule (Lemma \ref{lema22}), we obtain
	\begin{equation*}
		0\in \partial \left(g(\cdot)+\delta_{\Omega}(\cdot)+\frac{1}{k}\|\cdot-y_k\|\right)(y_k).
	\end{equation*}
	By the Lipschitz property of the function $\|\cdot-y_k\|$ and the sum rule (Lemma \ref{sum-rule}), one gets
	\begin{equation}\label{equa-7}
		0\in \partial (g(\cdot)+\delta_{\Omega}(\cdot))(y_k)+\frac{1}{k}\partial(\|\cdot-y_k\|)(y_k).
	\end{equation}
	We now show that condition \eqref{CQ-infinity} implies that there exists $R>0$ such that
	\begin{equation}\label{equa-8}
		\partial^\infty g(x)\cap (-N(x; \Omega))=\{0\} \ \ \text{when}\ \ \|x\|>R.
	\end{equation} 
	Indeed, if otherwise, there exist sequences $z_k\in\R^n$ and $u_k\in \partial^\infty g(z_k)\cap (-N(z_k; \Omega))$ such that $\|z_k\|\to\infty$ and $u_k\neq 0$ for all $k\in\N$. Clearly, 
	$$\frac{u_k}{\|u_k\|}\in \partial^\infty g(z_k)\cap (-N(z_k; \Omega)).$$ 
	By passing to a subsequence if necessary we may assume that $\frac{u_k}{\|u_k\|}$ converges to some $u$ with $\|u\|=1$. Thus,  by Proposition \ref{pro42} we get $u\in \partial^\infty g(\infty)\cap (-N(\infty; \Omega))$, a contradiction. 
	
	By \eqref{equa-8}, \eqref{equa-7} and the sum rule, we have
	\begin{equation}\label{equa-9}
		0\in \partial g(y_k)+\partial \delta_\Omega(y_k)+\frac{1}{k}\partial(\|\cdot-y_k\|)(y_k).
	\end{equation}
	We note here that $\partial \delta_\Omega(y_k)=N(y_k; \Omega)$ and  $\partial(\|\cdot-y_k\|)(y_k)=\mathbb{B}$.  Thus \eqref{equa-9} means that 
	\begin{equation*}
		0\in \partial g(y_k)+N(y_k; \Omega)+\frac{1}{k}\mathbb{B}.
	\end{equation*}
	Hence, for each $k\in\N$ there exist $u_k\in\partial g(y_k)$ and $v_k\in N(y_k; \Omega)$ such that
	\begin{equation}\label{equa-10}
		\|u_k+v_k\|\leq \frac{1}{k}
	\end{equation}
	and so $\lim_{k\to\infty}(u_k+v_k)=0$. 
	
	We have the following two cases.
	
	{\em Case 1. The sequence $u_k$ is bounded}. Then by \eqref{equa-10}, the sequence $v_k$ is also bounded. By passing to a subsequence if necessary we may assume that $u_k\to u$ and $v_k\to v$ as $k\to\infty$. Thus we arrive at
	\begin{equation*}
		u\in\partial g(\infty), v\in N(\infty; \Omega) \ \ \text{and} \ \ u+v=0,
	\end{equation*} 
	which contradicts \eqref{regular-infinity}.  
	
	{\em Case 2. The sequence $u_k$ is unbounded}. By passing to a subsequence if necessary we may assume that $u_k\to\infty$ as $k\to\infty$ and so is $v_k$. Furthermore, by \eqref{equa-10} we have
	\begin{equation*}
		\|v_k\|\leq \|u_k+v_k\|+\|u_k\|\leq \frac{1}{k}+\|u_k\|.
	\end{equation*}
	Hence,
	\begin{equation*}
		\frac{\|v_k\|}{\|u_k\|}\leq 1+\frac{1}{k\|u_k\|}.
	\end{equation*}
	This implies that the sequence $\frac{v_k}{\|u_k\|}$ is bounded and $\frac{v_k}{\|u_k\|}\in N(y_k;\Omega)$ for all $k\in \N$. Hence, by passing to a subsequence if necessary we may assume that $\frac{u_k}{\|u_k\|}\to u$ and $\frac{v_k}{\|u_k\|}\to v$ as $k\to\infty$. By Proposition \ref{pro42}, $u\in \partial^\infty g(\infty)$.  Clearly, $u\neq 0$, $v\in N(\infty; \Omega)$ and $u+v=0$, which contradicts \eqref{CQ-infinity}. The proof is complete.
\end{proof}
\begin{remark}\rm  By Proposition \ref{pro52}, the condition \eqref{CQ-infinity} holds automatically when $g$ is Lipschitz at infinity or $\Omega=\R^n$.
\end{remark}
The following simple example is to illustrate Theorem  \ref{theorem-3.2}.
\begin{example}\rm  Consider the constraint \eqref{constraint-set} with $\Omega=\R^2$, $g(x, y)=x^2+y^2$ for all $(x, y)\in \R^2$. Then $S=\{(0,0)\}$.  It is easy to see that $\partial g(\infty)=\emptyset$. Hence, by Theorem \ref{theorem-3.2}, $S$ has an error bound at infinity. However, we can check that $S$ has no error bound of Hoffman's type, that is, there exists $\tau>0$ such that
	\begin{equation}\label{Hoffman}
		d((x,y); S)\leq \tau [g(x,y)]_+ \ \ \forall (x,y)\in\R^2;
	\end{equation}
	see, for example, \cite{Hoffman_52,Luo,Ngai-2005}. Indeed, if otherwise, let $(x_k, y_k)=(0,\frac{1}{k})$. Then by \eqref{Hoffman}, we obtain 
	$$\frac{1}{k}\leq \tau\frac{1}{k^2} \ \ \forall k\in\N,$$
	a contradiction.
\end{example}

We now apply Theorem \ref{theorem-3.2} to  constraint systems described not by a single inequality but possibly by finitely many inequalities.
\begin{theorem} Let $S$ be a constraint set defined by
	\begin{equation}\label{equa-21}
		S:=\{x\in \Omega\;:\; g_i(x)\leq 0, i=1, \ldots, m\}
	\end{equation} 
	where $g_i\colon\R^n\to\R$, $i\in I:=\{1, \ldots, m\}$, are l.s.c. functions, $\Omega$ is an unbounded closed subset in $\R^n$ such that $\Omega\cap (\cap_{i\in I}\,\dom\, g_i)$ is unbounded. If the following conditions hold
	\begin{equation}\label{equa-11}
		\left[u_1+\ldots+u_m+v=0, u_i\in\partial^{\infty} g_i(\infty), v\in N(\infty; \Omega)\right] \Rightarrow u_i=v=0 \ \ \forall i\in I,
	\end{equation} 
	and  
	\begin{equation}\label{equa-12}
		\nexists \lambda\in\Delta_m \ \ \text{such that}\ \   0\in \sum_{i=1}^m\lambda_i\circ\partial g_i(\infty) +N(\infty; \Omega)
	\end{equation}
	then $S$ has an error bound at infinity, i.e., there exist $\alpha>0$ and $R>0$ such that
	\begin{equation}\label{equa-14}
		d(x; S)\leq \alpha\sum_{i=1}^m[g_i(x)]_+ \ \ \forall x\in\Omega\ \ \text{with}\ \ \|x\|>R.
	\end{equation}
\end{theorem}  
\begin{proof} Let $g$ be the maximum function defined by $g(x):=\max\,\{g_i(x)\;:\; i\in I\}$. Then 
	$$S=\{x\in\Omega\;:\; g(x)\leq 0\}.$$
	It follows from \eqref{equa-11} that
	\begin{equation*} 
		\left[u_1+\ldots+u_m=0, u_i\in\partial^{\infty} g_i(\infty)\right] \Rightarrow u_i=0 \ \ \forall i\in I.
	\end{equation*} 
	By Proposition \ref{pro-2.15}, one has
	\begin{equation*}
		\partial g(\infty) \subset \bigcup\left\{\sum_{i=1}^m\lambda_i\circ\partial f_i(\infty)\;:\; \lambda\in \Delta_m\right\}
	\end{equation*}
	This and \eqref{equa-12} imply that $0\notin \partial g(\infty) +N(\infty; \Omega)$. Thus, by  Theorem \ref{theorem-3.2}, there exist $\alpha>0$ and $R>0$ such that
	\begin{equation}\label{equa-15}
		d(x; S)\leq \alpha [g(x)]_+ \ \ \forall x\in\Omega \ \ \text{with}\ \ \|x\|>R.
	\end{equation}
	It is clear that 
	\begin{equation*}
		[g(x)]_+\leq  \sum_{i=1}^m[g_i(x)]_+ \ \ \forall x\in \R^n
	\end{equation*}
	which  together with \eqref{equa-15} implies \eqref{equa-14}. The proof is complete.   
\end{proof}
When $g_i$, $i\in I$, are Lipsschitz at infinity, we have the following result.
\begin{theorem} Let $S$ be given as in \eqref{equa-21}, 
	where $\Omega$ is a nonempty and unbounded closed subset in $\R^n$,  $g_i\colon\R^n\to\R$, $i\in I$, are l.s.c. functions and Lipschitz at infinity. If 
	\begin{equation*} 
		0\notin \mathrm{co}\,\{\partial g_i(\infty)\;:\; i\in I\}+N(\infty; \Omega)
	\end{equation*}
	then $S$ has an error bound at infinity, i.e., there exist $\alpha>0$ and $R>0$ such that
	\begin{equation*} 
		d(x; S)\leq \alpha\sum_{i=1}^m[g_i(x)]_+ \ \ \forall x\in\Omega \ \ \text{with}\ \ \|x\|>R.
	\end{equation*}
\end{theorem}
\begin{proof} Since  $g_i$, $i\in I$, are Lipsschitz at infinity, so is for its maxima $g$. Hence, the  condition \eqref{CQ-infinity} is satisfied. Then the desired result follows from Theorem \ref{theorem-3.2} and Proposition \ref{pro-2.15}.
\end{proof}

The next result gives a sufficient condition for the existence of error bound at infinity for a constraint system described by the pointwise minimum of finitely many  l.s.c. functions.
\begin{theorem} Let $S$ be given as in \eqref{equa-21}, 
	where $\Omega$ is a nonempty and unbounded closed subset in $\R^n$,  $g_i\colon\R^n\to\R$, $i\in I$, are l.s.c. functions and $g(x):=\min\,\{g_i(x)\;:\; i\in I\}$ for all $x\in\R^n$. Assume that the following condition
	\begin{equation}\label{equa-13}
		0\notin \bigcup\{\partial g_i(\infty)\;:\; i\in I\}+N(\infty; \Omega).
	\end{equation}
	is satisfied. Then the set $S$ has an error bound at infinity, i.e., there exist $\alpha>0$ and $R>0$ such that
	\begin{equation*} 
		d(x; S)\leq \alpha [g(x)]_+ \ \ \forall x\in\Omega \ \ \text{with}\ \ \|x\|>R.
	\end{equation*}
\end{theorem}
\begin{proof}
	By Proposition \ref{pro-2.16}, the condition \eqref{equa-13} implies that $0\notin \partial g(\infty)+N(\infty; \Omega)$. Thus the desired result follows directly from Theorem \ref{theorem-3.2}.
\end{proof}
\section{Optimality conditions at infinity}\label{Section 4}
In this section, by using the existence of error bounds at infinity, we derive an upper estimate the normal cone at infinity of the constraint set. This result is instrumental to derive  necessary optimality conditions at infinity for constrained optimization problems that have no solution. 

We first derive an upper estimate for the normal cone at infinity of a given unbounded subset via the subdifferential of the distance
function to the set in question.        
\begin{proposition}
	Let $S$ be a nonempty, closed and unbounded subset in $\R^n$. Then we have
	\begin{equation*}
		N(\infty; S)\subset\R_+\partial d(\infty; S).
	\end{equation*} 
\end{proposition}
\begin{proof}
	Let $u\in N(\infty; S)$. Clearly, if $u=0$, then $u\in \R_+\partial d(\infty; S)$. If $u\neq 0$, then there exist  sequences $x_k\in S$ and $u_k\in \widehat{N}(x_k; S)\setminus\{0\}$ such that $x_k\to\infty$ and $u_k\to u$  as $k\to\infty$. Clearly, $\frac{u_k}{\|u_k\|}\to \frac{u}{\|u\|}$ and by \cite[Corollary 1.96]{Mordukhovich2006} we have
	\begin{equation*}
		\frac{u_k}{\|u_k\|}\in \widehat{N}(x_k; S)\cap\mathbb{B}=\widehat{\partial}d(x_k; S). 
	\end{equation*}
	Hence, $\frac{u}{\|u\|}\in \partial d(\infty; S)$ and so $u\in \R_+\partial d(\infty; S)$. The proof is complete.
\end{proof}
\begin{remark}\rm  In \cite[Theorem 1.97]{Mordukhovich2006}, Mordukhovich showed that
	\begin{equation*}
		N(x; S)=\R_+\partial d(x; S) \ \ \forall x\in S.
	\end{equation*}
	However, this equality does not hold at infinity. For example, let 
	$$S=\{x=(x_1, x_2)\;:\; x_1\in\R, x_2=x_1^2\}.$$
	Then it is easily to check that $N(\infty; S)=\R\times\{0\}$. Furthermore, by \cite[Theorem 1.33]{Mordukhovich2018}, we have
	\begin{equation*}
		\partial d(x; S)=
		\begin{cases}
			N(x; S)\cap \mathbb{B} \ \ &\text{if} \ \ x\in S,
			\\
			\frac{x-\Pi_S(x)}{d(x; S)} \ \ &\text{otherwise}.
		\end{cases}
	\end{equation*} 
	Thus
	\begin{equation*}
		\partial d(\infty; S)=(N(\infty; S)\cap \mathbb{B})\cup\Limsup_{x \to \infty,\, x\notin S}\frac{x-\Pi_S(x)}{d(x; S)}.
	\end{equation*}
	Take $x^k=(0, -k)$, $k\in\N$, then 
	$$\lim_{k\to\infty} \frac{x^k-\Pi_S(x^k)}{d(x^k; S)}=(0, -1)\in \partial d(\infty; S)$$ 
	and $(0, -1)\notin N(\infty; S)$.
\end{remark}

The following result gives an upper estimate for constraint systems that have an error bound at infinity.
\begin{proposition}\label{Pro-43} Let the constraint set $S$ be defined by
	\begin{equation}\label{equa-17}
		S:=\{x\in\R^n\;:\; g(x)\leq 0\},
	\end{equation}  
	where $g\colon\R^n\to\overline{\mathbb{R}}$ is a l.s.c. function. Assume that $S$ has an error bound at infinity. Then we have
	\begin{equation*} 
		N(\infty; S)\subset \bigcup\{\lambda\circ \partial g(\infty)\;:\; \lambda\geq 0\}.
	\end{equation*}
	If in addition $g$ is Lipschitz at infinity, then
	\begin{equation}\label{equa-22} 
		N(\infty; S)\subset \bigcup\{\lambda\partial g(\infty)\;:\; \lambda\geq 0\}.
	\end{equation}
\end{proposition}
\begin{proof} By assumption, there exist $\alpha>0$ and $R>0$ such that
	\begin{equation*}
		d(x; S)\leq \alpha [g(x)]_+ \ \ \forall x\in\R^n\setminus\mathbb{B}_R.
	\end{equation*}
	Put $f(x):=\alpha[g(x)]_+$ for all $x\in\R^n$. Then the function $f$ possess the following properties
	\begin{align*}
		&f(x)=0\ \ \forall x\in S,
		\\
		&d(x, S)\leq f(x) \ \ \forall x\in\R^n\setminus\mathbb{B}_R. 
	\end{align*} 
	Hence, by definition of Fr\'echet subdiffential and \cite[Corollary 1.96]{Mordukhovich2006}, we have
	\begin{equation*}
		\widehat{N}(x; S)\cap\mathbb{B}=\widehat{\partial}d(x, S)\subset \widehat{\partial} f(x) \ \ \forall x\in S\setminus\mathbb{B}_R. 
	\end{equation*}
	This implies that
	\begin{equation*}
		\Limsup_{x \xrightarrow{S} \infty}[\widehat{N}(x; S)\cap\mathbb{B}]\subset \Limsup_{x \xrightarrow{S} \infty} \widehat{\partial} f(x)\subset \partial f(\infty). 
	\end{equation*}
	We claim that 
	\begin{equation*}
		N(\infty; S)\cap\mathbb{B}=\Limsup_{x \xrightarrow{S} \infty}[\widehat{N}(x; S)\cap\mathbb{B}]
	\end{equation*}
	and so
	\begin{equation}\label{equa-16}
		N(\infty; S)\cap\mathbb{B}\subset \partial f(\infty).  
	\end{equation}
	Indeed, it is easy to see that
	\begin{equation*}
		\Limsup_{x \xrightarrow{S} \infty}[\widehat{N}(x; S)\cap\mathbb{B}]\subset N(\infty; S)\cap\mathbb{B}. 
	\end{equation*}
	Now take any $u\in N(\infty; S)\cap\mathbb{B}$. If $u=0$, then  $u\in \Limsup_{x \xrightarrow{S} \infty}[\widehat{N}(x; S)\cap\mathbb{B}]$. Otherwise, there exist sequences $x_k\in S$, $u_k\in \widehat{N}(x_k; S)$ such that $x_k\to\infty$ and $u_k\to u$ as $k\to\infty$. Since $u\in \mathbb{B}$, we see that
	\begin{equation*}
		\frac{\|u\|}{\|u_k\|}.u_k\in \widehat{N}(x_k; S)\cap\mathbb{B} \ \ \text{and}\ \ \frac{\|u\|}{\|u_k\|}.u_k \to u
	\end{equation*}
	and so $u\in \Limsup_{x \xrightarrow{S} \infty}[\widehat{N}(x; S)\cap\mathbb{B}]$, as required.  
	
	We now obtain from \eqref{equa-16} and Proposition \ref{pro-2.15} that
	\begin{align*}
		N(\infty; S)=\R_+[N(\infty)\cap\mathbb{B}]&\subset \R_+\partial f(\infty)
		\\
		&=\R_+\big[(0, \alpha]\partial g(\infty)\cup \partial^\infty g(\infty)\big]
		\\
		&=[\R_+\partial g(\infty)]\cup\partial^\infty g(\infty)
		\\
		&=\bigcup\{\lambda\circ \partial g(\infty)\;:\; \lambda\geq 0\}.
	\end{align*} 
	If $g$ is Lipschitz at infinity, then by Proposition \ref{pro52}, $\partial^\infty g(\infty)=0$ and so \eqref{equa-22} is valid. The proof is complete.
\end{proof}

We now apply the estimate \eqref{equa-22} to derive necessary optimality conditions at infinity for constrained optimization problems.  Let $f\colon\R^n\to\overline{\mathbb{R}}$ be a l.s.c. function and let $S$ be a nonempty and closed subset in $\R^n$.  Assume that the following conditions hold:

(A1) $\dom f\cap S$ is unbounded;

(A2) $f$ in bounded from below on $S$.

Consider the following minimization problem
\begin{equation}\label{problem}\tag{P}
	\text{minimize}\, f(x)\ \ \text{such that}\ \ x\in S.
\end{equation}
Let us recall necessary optimality conditions at infinity to problem \eqref{problem}. 
\begin{theorem}[{\cite[Theorem 6.1]{Kim-Tung-Son-23}}]\label{Necessary-Theorem} If $f$ does not attain its infimum on $S$ and the following condition holds
	\begin{equation}\label{equa-19}
		-\partial^\infty f(\infty)\cap N(\infty; S)=\{0\},
	\end{equation}
	then
	\begin{equation*}
		0\in\partial f(\infty)+N(\infty; S).
	\end{equation*}
\end{theorem}  
The following result gives necessary optimality conditions at infinity to problem \eqref{problem} when the constraint set $S$ is given as in \eqref{equa-17}. 
\begin{theorem}\label{Necessary-Theorem-1} Consider the problem \eqref{problem} with $S$ defined as in \eqref{equa-17}. Assume that conditions (A1), (A2) and the following condition hold
	\begin{equation}\label{equa-18}
		-\partial^\infty f(\infty)\cap \bigg[\bigcup\{\lambda\circ \partial g(\infty)\;:\; \lambda\geq 0\}\bigg]=\{0\}. 
	\end{equation}
	If $f$ does not attains its infimum on $S$ and $S$ has an error bound at infinity, then there exists $\lambda\geq 0$ such that
	\begin{equation*}
		0\in\partial f(\infty)+\lambda\circ \partial g(\infty).
	\end{equation*}
	If in addition $g$ is Lipschitz at infinity, then there exists $\lambda\geq 0$ such that
	\begin{equation*}
		0\in\partial f(\infty)+\lambda \partial g(\infty).
	\end{equation*} 
\end{theorem}
\begin{proof} By assumptions and Proposition \ref{Pro-43}, we have
	\begin{equation}\label{equa-20} 
		N(\infty; S)\subset \bigcup\{\lambda\circ \partial g(\infty)\;:\; \lambda\geq 0\}.
	\end{equation}
	This and \eqref{equa-18} imply that the condition \eqref{equa-19} is satisfied. The desired conclusions follow  from \eqref{equa-20} and Theorem \ref{Necessary-Theorem}.  
\end{proof}
\begin{remark}\rm 
	If $f$ and $g$ are Lipschitz at infinity, then the condition \eqref{equa-18} holds automatically.   
\end{remark}
The following example is designed to illustrate Theorem \ref{Necessary-Theorem-1}.
\begin{example}\rm  Consider problem \eqref{problem} with $f(x)=e^x+\frac{1}{|x|+1}$, $g(x)=x$ for all $x\in \R$ and $S=\{x\in\R\;:\; g(x)\leq 0\}=\R_-$. Clearly, $g$ is Lipschitz at infinity and $\partial g(\infty)=1$. Hence, by Theorem \ref{theorem-3.2}, $S$ has an error bound at infinity. An easy computation shows that $\partial^\infty f(\infty)=\R_+$ and so the condition \eqref{equa-18} is satisfied. Furthermore, the function $f$ is bounded from below on $S$ but does not attain a minimum. Thus there exists $\lambda\geq 0$ such that 
	$$0\in\partial f(\infty)+\lambda \partial g(\infty).$$
\end{example}

The following result is deduced from Theorem \ref{Necessary-Theorem-1} and Proposition \ref{pro-2.15}.
\begin{corollary}
	Consider the problem \eqref{problem} with $S$ defined as in \eqref{equa-21} with $\Omega=\R^n$. Assume that conditions (A1), (A2), \eqref{equa-11} and the following condition  hold
	\begin{equation*} 
		-\partial^\infty f(\infty)\cap \bigg[\bigcup\Big\{\sum_{i=1}^m\lambda_i\circ \partial g_i(\infty)\;:\; \lambda\in\R^m_+\Big\}\bigg]=\{0\}. 
	\end{equation*}
	If $f$ does not attains its infimum on $S$ and $S$ has an error bound at infinity, then there exists $\lambda\in\R^m_+$ such that
	\begin{equation*}
		0\in\partial f(\infty)+\sum_{i=1}^m\lambda_i\circ \partial g_i(\infty).
	\end{equation*}
	If in addition $g_1, \ldots, g_m$ are  Lipschitz at infinity, then there exists $\lambda\in\R^m_+$ such that
	\begin{equation*}
		0\in\partial f(\infty)+\sum_{i=1}^m\lambda_i \partial g_i(\infty).
	\end{equation*}
\end{corollary}
\begin{proof} Let $g$ be the function defined by $g(x):=\max\,\{g_i(x)\;:\; i=1, \ldots, m\}$. Then the desired result follows from Theorem \ref{Necessary-Theorem-1} and Proposition \ref{pro-2.15}.
\end{proof}
%====================================================================
%\section{Conclusions}\label{Conclusion}

\section*{Acknowledgments} 
A part of this work was done while the author was visiting Department of Applied Mathematics, Pukyong National University, Busan, Korea in October 2023. The  author would like to thank the department for hospitality and support during their stay.

\end{document}